\documentclass[reqno, 12pt]{amsart}
\usepackage{amsmath,amsthm,amsfonts,amssymb}
\usepackage[cp1251]{inputenc}
   \usepackage[english]{babel}
\date{}
\newtheorem{theorem}{Theorem}[section]
\newtheorem{lemma}[theorem]{Lemma}
\newtheorem{corollary}[theorem]{Corollary}
\newtheorem{remark}[theorem]{Remark}
\newtheorem{proposition}[theorem]{Proposition}
\newtheorem{example}[theorem]{Example}
\numberwithin{equation}{section}

\begin{document}

\centerline{\sc Zvonkin's transform and the regularity of solutions to double}

\centerline{\sc divergence form elliptic equations}

\vspace*{0.2cm}

\centerline{\sc Vladimir I. Bogachev, Michael R\"ockner, Stanislav V. Shaposhnikov}

\vspace*{0.3cm}

{\small
{\bf Abstract.}
We study qualitative properties of  solutions to double divergence form elliptic equations
(or stationary Kolmogorov equations)  on~$\mathbb{R}^d$.
It is shown that the Harnack inequality holds for nonnegative solutions
 if the diffusion matrix $A$ is nondegenerate and satisfies the Dini mean oscillation condition and the drift coefficient $b$
is locally integrable to a power $p>d$.
We establish new estimates for the $L^p$-norms of solutions and
obtain a generalization of the known theorem of Hasminskii on the existence
of a probability solution to the stationary Kolmogorov equation to the case where the matrix $A$ satisfies Dini's
condition   or belongs to the class VMO. These results are based on a new analytic version of Zvonkin's transform
of the drift coefficient.
}



\section{Introduction}

We study qualitative properties of  solutions to the double divergence form elliptic equation
(or the stationary Kolmogorov equation)
\begin{equation}\label{eq1}
\partial_{x_i}\partial_{x_j}\bigl(a^{ij}\varrho\bigr)-\partial_{x_i}\bigl(b^i\varrho\bigr)=0
\end{equation}
on an open set $\Omega\subset\mathbb{R}^d$.
The matrix $A=(a^{ij})$ is supposed to be symmetric and positive definite,
$a^{ij}$ and $b^i$ are  Borel functions.
Set
$$
L\varphi=a^{ij}\partial_{x_i}\partial_{x_j}\varphi+b^i\partial_{x_i}\varphi,
\quad L^{*}\varphi=\partial_{x_i}\partial_{x_j}\bigl(a^{ij}\varphi\bigr)-\partial_{x_i}\bigl(b^i\varphi\bigr).
$$
Then equation (\ref{eq1}) can be written in a shorter form
$$
L^{*}\varrho=0.
$$
A function $\varrho\in L^1_{loc}(\Omega)$ is a solution to equation (\ref{eq1}) if
$$
a^{ij}\varrho, b^i\varrho\in L^1_{loc}(\Omega)
$$
and for every function $\varphi\in C_0^{\infty}(\Omega)$
the equality
$$
\int_{\Omega}L\varphi(x)\varrho(x)\,dx=0
$$
is fulfilled. A nonnegative solution $\varrho$ to the  Kolmogorov equation~{\rm (\ref{eq1})} satisfying the condition
$$
\int_{\Omega}\varrho(x)\,dx=1
$$
is called a probability solution.

An important example of a double divergence form elliptic equation is delivered by the stationary
 Kolmogorov  equation for invariant measures of a diffusion process.
Various properties of solutions to such equations were studied by many authors.
The principal questions are

1) the existence of solutions,
especially, of probability solutions,

2) the existence of solution densities and their properties such as local
boundedness, continuity and Sobolev differentiability,

3) local separation of
densities from zero, that is, certain forms of the Harnack inequality.

In case of locally Lipschitz coefficients the existence of a probability solution
is given by the  classical theorem of Hasminskii \cite{Has}
under the existence of a Lyapunov function. This theorem was generalized in \cite{BR}
(see also further generalizations in \cite[Chapter 2]{book}, \cite{BRSH12}, \cite{Trut1}), where either the
diffusion coefficient is  nondegenerate and  locally Sobolev with the order of integrability higher than dimension along with
the same local integrability of the drift coefficient or both the diffusion and drift coefficients are continuous.
It was shown in  \cite{BKR} that the solution density is locally Sobolev in the
first case and its continuous version is locally separated from  zero.
It was proved in \cite{Sjogan73} and \cite{Sjogan75}  that in the case where the matrix
$A=(a^{ij})$ is nondegenerate and satisfies  Dini's condition  and the coefficients $b^i$ are bounded,
the solution has a continuous version, and when the coefficients  $a^{ij}$ are H\"older continuous, then the solution has a H\"older continuous version.
These results have been generalized in \cite{BSH17} to the case of integrable $b^i$.
Analogous results have been obtained in \cite{DongKim} and \cite{DongEscKim}  under the assumption
that the matrix $A$ satisfies the Dini mean oscillation condition, which is weaker than the classical  Dini  condition.
Note also the paper \cite{LePimen}, where some additional regularity of solutions along level sets has been established.
In the papers  \cite{Bauman1}, \cite{Bauman2} some interesting counter-examples were constructed and the so-called renormalized
solutions were studied, in particular, an example was constructed of a positive definite and continuous diffusion matrix $A$
for which the equation $\partial_{x_i}\partial_{x_j}(a^{ij}\varrho)=0$
has a locally unbounded solution.
The Harnack inequality for double divergence form equations with the matrix $A$
belonging to the Sobolev class with a sufficiently high integrability exponent
is a corollary of the Harnack inequality for divergence form
elliptic equations (see \cite[Chapter 3]{book}). However, in case of merely H\"older continuous
matrix $A$ the double divergence form equation cannot be reduced to a divergence form equation,
moreover, the classical results about the regularity of solutions
to divergent form elliptic equations are not true for solutions to double divergence form equations.
In the case where the matrix $A$ satisfies  Dini's  condition, the Harnack inequality was
obtained in \cite{Mamedov} for $b=0$, and for any bounded drift $b$ it was established in \cite{BSH17}.
The proof consisted in obtaining certain inequalities for solutions generalizing classical mean value
theorems and heavily used the boundedness of $b$.
Another way of proving the Harnack inequality for $b=0$
was suggested in \cite{DongEscKim},
where the reasoning employs  estimates for the modulus of continuity of the solution and some properties of renormalized solutions
from the paper \cite{Esc} considering double divergence form equations
without first order terms.
In \cite{BKR} and \cite{BSH17} (see also \cite[Chapter 1]{book}) the
integrability of solutions was investigated in some cases when the  diffusion matrix does not satisfy Dini's condition.
In particular, it was shown that if $A$ belongs to the  class $VMO$ and the coefficient $b$
is locally integrable to some power $p>d$, then the solution belongs to all $L^p_{loc}(\Omega)$.
In spite of a considerable number of papers devoted to double divergence form
elliptic equations, the answers to the following questions, certain specifications of general
problems 1 --- 3 mentioned above,  have remained open so far:

$\bullet$ What are optimal conditions for the Harnack inequality for nonnegative solutions to double divergence form equations?
In particular, does the Harnack inequality hold for $A$ satisfying Dini's condition and
an unbounded locally  Lebesgue integrable drift $b$?

$\bullet$ What are optimal conditions for a high  local integrability of solutions?
In particular, the dependence of
the  integrability of the solution on the modulus of continuity of the matrix $A$ has not been studied.

$\bullet$ What are optimal conditions for the existence and uniqueness of a probability solution to the stationary Kolmogorov
equation?

Here we obtain new results related to these questions. (i) We prove that the Harnack inequality holds for nonnegative solutions on $\mathbb{R}^d$
 if the matrix $A$ is nondegenerate and satisfies the Dini mean oscillation condition and the  coefficient $b$
is locally integrable to a power $p>d$.

(ii)~We establish new estimates for the $L^p$-norms of solutions and
obtain sufficient conditions for the local
exponential integrability. Note that it was asserted in \cite{BSH17}
that in the case of a locally bounded coefficient $b$ and a nondegenerate matrix $A$
of class $VMO$ the solution is locally  exponentially integrable. However, the justification given there contains a gap, namely,
a wrong dependence on $p$ of the
constant in a priori $L^p$-estimates of second derivatives of solutions to the equation ${\rm tr}(AD^2u)=f$.
In the general case, the dependence of the constant on $p$ is influenced by the modulus of
continuity of $A$. In the present paper we derive an  estimate that takes the modulus of continuity of $A$ into account.

(iii) Finally, an important new result of our paper is a generalization of the known theorem of Hasminskii on the existence
of a probability solution to the stationary Kolmogorov equation to the case where the matrix $A$ satisfies Dini's
condition   or belongs to the class VMO.
Results on existence of positive or probability
solutions to the stationary  Kolmogorov equation in case of irregular coefficients
are useful for constructing  diffusion processes (see \cite{Trut1}). We also discuss
uniqueness of probability solutions and their probabilistic interpretation.

These results are obtained with the aid of a new approach to the study of regularity of
solutions to double divergence form  equations based on Zvonkin's transform, well known in the theory of diffusion processes,
which applies for smoothing the drift coefficient
(more precisely, we deal with its elliptic version, in the original paper \cite{Zvonkin} this transform was used for parabolic equations).
In recent years Zvonkin's transform has been applied for the study of
diffusion processes with generalized coefficients (see \cite{FlandRus}, \cite{RZ21}, \cite{ZZ18}).
In this paper we apply Zvonkin's transform not to random processes,
but to solutions of the Kolmogorov equation, moreover,
we do not assume any connection of solutions with diffusion processes. It is shown below that with the aid of a suitable change
of coordinates an integrable drift can be transformed into a continuously differentiable drift such  that the new  diffusion matrix
enables us to apply known results about  regularity of solutions. This leads to substantial generalizations of some results and
simplification of proofs of other results on regularity of
solutions to  double divergence form  elliptic equations. Note that change of coordinates has proved
to be also useful in the study of uniqueness of solutions (see \cite{Tih}).

In \S2 we construct Zvonkin's transform in the analytic setting,
in \S3 we study the regularity of solutions,
and in \S4 we apply our results for  proving existence and uniqueness of probability solutions.

\section{Zvonkin's transform}

We first illustrate our approach by example of a smooth change of coordinates.

Let $\Phi\colon\mathbb{R}^d\to\mathbb{R}^d$ be a diffeomorphism of  class $C^2$ and $\Psi=\Phi^{-1}$.
Set
$$
q^{km}(y)=a^{ij}(\Psi(y))\partial_{x_i}\Phi^k(\Psi(y))\partial_{x_j}\Phi^m(\Psi(y)),
$$
$$
h^k(y)=a^{ij}(\Psi(y))\partial_{x_i}\partial_{x_j}\Phi^k(\Psi(y))+b^i(\Psi(y))\partial_{x_i}\Phi^k(\Psi(y)),
$$
$$
\sigma(y)=\varrho(\Psi(y))|\det \nabla\Psi(y)|.
$$
Then on the domain $\Omega'=\Phi(\Omega)$ the function $\sigma$ satisfies the equation $\mathcal{L}^{*}\sigma=0$
with the operator
$$
\mathcal{L}\varphi(y)=q^{km}(y)\partial_{y_k}\partial_{y_m}\varphi(y)+h^{k}(y)\partial_{y_k}\varphi(y).
$$
Observe that the double divergence structure of the equation does not change.
One can construct a  mapping $\Phi$ of the form $\Phi(x)=x+u(x)$, where $u=(u^1,\ldots,u^d)$,
such that $u^k$ is a solution to the elliptic equation
$$
Lu^k-\lambda u^k=-b^k.
$$
For $\lambda$ sufficiently large, this equation possesses a solution for which $\Phi$ is a
diffeomorphism. Using $\Phi$  to change variables we obtain a new drift coefficient
 $h(y)=\lambda u(\Psi(y))$. It turns out that under fairly general assumptions about the coefficients
(see below) the vector field $h$ is continuously differentiable and the
regularity of the matrix $(q^{km})$ is not worse than that of the original matrix $(a^{ij})$.
This circumstance enables us after the change of coordinates to apply the known results on the regularity of solutions
and obtain the desired properties for the function~$\sigma$,
hence also for original solution~$\varrho$.
The main difficulty consists in constructing the mapping~$u$.
Here we employ some recent  results of N.V.~Krylov
on the solvability of elliptic equations in Sobolev spaces
(see \cite{Kr2007} and \cite{Krbook}).

Under broad assumptions, we construct a diffeomorphism, which will be called Zvonkin's transform.

Throughout this section we assume that the  following conditions are fulfilled.

\, $\rm\bf H_a$ \, The coefficients $a^{ij}$ are defined on all of $\mathbb{R}^d$ and for some constant $m>0$ and all $x\in\mathbb{R}^d$
the following inequalities hold:
$$
m\cdot{\rm I}\le A(x)\le m^{-1}\cdot{\rm I}.
$$

\, $\rm\bf VMO$ \, The coefficients $a^{ij}$ belong to the class $VMO$, that is,
there exists a continuous increasing function $\omega$ on $[0, +\infty)$ such that $\omega(0)=0$ and
$$
\sup_{z\in\mathbb{R}^d}r^{-2d}\int_{B(z, r)}\int_{B(z, r)}|a^{ij}(x)-a^{ij}(y)|\,dxdy\le\omega(r), \quad r>0.
$$

\, $\rm\bf H_b$ \, $b\in L^{d+}_{loc}$, which means that for every ball $B$ there is a  number $p=p(B)>d$ such
that the restriction of $|b|$ to $B$ belongs to $L^{p}(B)$.

\vspace*{0.2cm}

The assumption  that the  diffusion matrix satisfies the aforementioned conditions on the whole space $\mathbb{R}^d$ does not restrict
the generality of our considerations, although the equation will be considered on a domain.
Moreover, in many problems it is useful to have global changes of variables,
but not local. For our purposes of proving  local Harnack inequalities or the continuity of solutions it suffices
to extend the coefficients on the whole space with preservation of the required conditions. The drift coefficient can be extended by
zero outside a fixed ball~$B$ and the  diffusion coefficient can be extended by the formula $\psi A+(1-\psi)I$
with a smooth function $\psi$ that equals $1$ on~$B$ and $0$ outside a larger ball.
Of course, it is important here that the equation holds on a ball, but not on the whole space.

Let $B(x_0, 4R)\subset\Omega$ and $\beta(x)=b(x)$ if $x\in B(x_0, 4R)$ and $\beta(x)=0$ if $x\notin B(x_0, 4R)$.
Then $\beta\in L^p(\mathbb{R}^d)$ and
$$
\|\beta\|_{L^p(\mathbb{R}^d)}=\|b\|_{L^p(B(x_0, 4R))}.
$$
Let $1\le k\le d$. Let us consider on $\mathbb{R}^d$ the elliptic equation
\begin{equation}\label{eq2}
{\rm tr}(AD^2u)+\langle\beta, \nabla u\rangle-\lambda u=-\beta^k, \quad \lambda>0.
\end{equation}

\begin{proposition}\label{prop1}
For every $\delta>0$ there exists $\lambda>0$ such that for every $1\le k\le d$
equation {\rm (\ref{eq2})} has a solution $u\in C^1(\mathbb{R}^d)\cap W^{p, 2}(\mathbb{R}^d)$ for which
$$
\sup_{x\in\mathbb{R}^d}|\nabla u(x)|\le \delta, \quad \|u\|_{W^{2, p}(\mathbb{R}^d)}\le M,
$$
where the constant $M$ depends only on $d$, $\nu$, $\omega$ and $\|b\|_{L^p(B(x_0, 4R))}$.
\end{proposition}
\begin{proof}
According to \cite[Chapter 6, Section 4, Theorem 1]{Krbook}, there exist numbers $\lambda_0>0$ and $N_0$ such that for all $\lambda>\lambda_0$
and every function $v\in W^{p, 2}(\mathbb{R}^d)$ we have the inequality
$$
\lambda\|v\|_{L^p(\mathbb{R}^d)}+\|v\|_{W^{p,2}(\mathbb{R}^d)}\le N_0\|{\rm tr}(AD^2v)-\lambda v\|_{L^p(\mathbb{R}^d)}.
$$
For every function $f\in L^p(\mathbb{R}^d)$ there exists a unique solution $v\in W^{2, p}(\mathbb{R}^d)$
of the equation
$$
{\rm tr}(AD^2v)-\lambda v=f.
$$
Since $p>d$, by the   embedding theorem for every function $v\in W^{p, 2}(\mathbb{R}^d)$
one has the estimate
$$
\|\nabla v\|_{L^{\infty}(\mathbb{R}^d)}\le N_1\|v\|_{W^{p, 2}(\mathbb{R}^d)}.
$$
By \cite[Chapter 1, Section 5, Corollary 2]{Krbook} there  exists a constant $N_2$ such that
$$
\|\nabla v\|_{L^p(\mathbb{R}^d)}\le N_2\|D^2v\|_{L^p(\mathbb{R}^d)}+N_2\|v\|_{L^p(\mathbb{R}^d)}.
$$
Therefore,
$$
\|\nabla v\|_{L^{\infty}(\mathbb{R}^d)}\le N_3\|D^2v\|_{L^p(\mathbb{R}^d)}+N_3\|v\|_{L^p(\mathbb{R}^d)}.
$$
Let $\varepsilon>0$. By a standard reasoning, replacing $x$ with $\varepsilon x$, we obtain the inequality
$$
\|\nabla v\|_{L^{\infty}(\mathbb{R}^d)}\le \varepsilon^{1-d/p}N_3\|D^2v\|_{L^p(\mathbb{R}^d)}
+\varepsilon^{-1-d/p}N_3\|v\|_{L^p(\mathbb{R}^d)}.
$$
Thus, we can assume that for every $\varepsilon>0$ there exists a constant $N_4=N_4(\varepsilon)$ for which
$$
\|\nabla v\|_{L^{\infty}(\mathbb{R}^d)}\le \varepsilon\|v\|_{W^{2, p}(\mathbb{R}^d)}+N_4\|v\|_{L^p(\mathbb{R}^d)}.
$$
Let us now estimate the expression
$$
\|\langle\beta, \nabla v\rangle\|_{L^p(\mathbb{R}^d)}.
$$
We have
$$
\|\langle\beta, \nabla v\rangle\|_{L^p(\mathbb{R}^d)}\le
N_4
\|\beta\|_{L^p(\mathbb{R}^d)}\|v\|_{L^p(\mathbb{R}^d)}+\varepsilon\|\beta\|_{L^p(\mathbb{R}^d)}\|v\|_{W^{2, p}(\mathbb{R}^d)}.
$$
Take $\varepsilon_0>0$ and $\lambda_1\ge \lambda_0$ such that
$$
\varepsilon_0 N_0\|\beta\|_{L^p(\mathbb{R}^d)}<1/2, \quad N_4N_0\|\beta\|_{L^p(\mathbb{R}^d)}<\lambda_1/2.
$$
Then for every $\lambda>\lambda_1$ and $v\in W^{p, 2}(\mathbb{R}^d)$ we have
$$
\lambda\|v\|_{L^p(\mathbb{R}^d)}+\|v\|_{W^{p, 2}(\mathbb{R}^d)}\le
2N_0\|{\rm tr}(AD^2v)+\langle\beta,\nabla v\rangle -\lambda v\|_{L^p(\mathbb{R}^d)}.
$$
This estimate remains unchanged if we replace $\beta$ by $t\beta$, where $t\in[0, 1]$.
Set
$$
L_tv=t\Bigl({\rm tr}(AD^2v)+\langle\beta,\nabla v\rangle-\lambda v\Bigr)+(1-t)\Bigl({\rm tr}(AD^2v)-\lambda v\Bigr),
\quad t\in[0, 1].
$$
The continuous  operators $L_t$ from $W^{2, p}(\mathbb{R}^d)$ to $L^p(\mathbb{R}^d)$ satisfy the condition
$$
\|L_tv\|_{L^p(\mathbb{R}^d)}\ge (2N_0)^{-1}\|v\|_{W^{2, p}(\mathbb{R}^d)}.
$$
Hence the standard method of continuation with respect to a parameter ensures the existence of a solution $u$ to equation (\ref{eq2}).
By the embedding theorem $u\in C^1(\mathbb{R}^d)$. In addition,
$$
\|u\|_{L^p(\mathbb{R}^d)}\le 2\lambda^{-1}N_0\|\beta\|_{L^p(\mathbb{R}^d)},
\quad \|u\|_{W^{2, p}(\mathbb{R}^d)}\le 2N_0\|\beta\|_{L^p(\mathbb{R}^d)}.
$$
It has been shown above that for every $\varepsilon>0$ we have
$$
\|\nabla u\|_{L^{\infty}(\mathbb{R}^d)}\le \varepsilon\|u\|_{W^{2, p}(\mathbb{R}^d)}+N_4\|u\|_{L^p(\mathbb{R}^d)},
$$
where the right-hand side is estimated from above by
$$
2\varepsilon N_0\|\beta\|_{L^p(\mathbb{R}^d)}+\lambda^{-1}2N_0N_4\|\beta\|_{L^p(\mathbb{R}^d)}.
$$
Therefore, taking $\varepsilon$ sufficiently small and $\lambda$ sufficiently large we can obtain the desired
estimate $\|\nabla u\|_{L^{\infty}(\mathbb{R}^d)}\le \delta$.
\end{proof}

\begin{remark}\rm
Since $u\in W^{p,2}(\mathbb{R}^d)$ with $p>d$, by the   embedding theorem $u\in C^{1+(1-d/p)}(\mathbb{R}^d)$, hence
the function $u$ is bounded and the derivatives $u_{x_i}$ satisfy the H\"older condition of order $1-d/p$.
Note also that our construction of $u$ does not use a special form of $\beta$: only
the condition $\beta\in L^p(\mathbb{R}^d)$ is needed.
\end{remark}

Let $u=(u^1,\ldots,u^d)$, where each $u^k$ is a  solution to equation (\ref{eq2}) from Proposition \ref{prop1}.
Below $u'$ and $\Phi'$ denote the Jacobi matrices of the mappings $u$ and $\Phi$.
Let us take a number $\delta$ from the hypotheses of Proposition \ref{prop1} such that for all $x\in\mathbb{R}^d$
the  inequality
$$
\|u'(x)\|\le \frac{1}{2} \quad \hbox{\rm and} \quad \tfrac{1}{2}\le \det\bigl(I+u'(x)\bigr)\le 2
$$
is fulfilled. Set
$$
\Phi(x)=x+u(x).
$$
We now establish some properties of the mapping $\Phi$.

\begin{proposition}\label{prop2}
{\rm (i)} The mapping $\Phi$ is a diffeomorphism of $\mathbb{R}^d$ of class $C^1$, moreover, the functions
$\partial_{x_i}\Phi^k$ are locally H\"older continuous.

{\rm (ii)} The inequalities
$$\frac{1}{2}\|x-y\|\le \|\Phi(x)-\Phi(y)\|\le 2\|x-y\|$$
hold.
\end{proposition}
\begin{proof}
Since  $u$ is a contracting mapping with the Lipschitz constant  $1/2$, the mapping $\Phi$ is a homeomorphism
and the stated inequalities hold. The inclusion  $u\in C^1(\mathbb{R}^d)$ yields  that $\Phi\in C^1(\mathbb{R}^d)$ and
the Jacobi matrix has the form $\Phi'=I+u'$. By the estimate $\|u'\|\le 1/2$ the matrix $\Phi'$ is invertible.
Therefore, $\Phi$ is a diffeomorphism.
\end{proof}

We recall that $B(x_0, 4R)\subset\Omega$.  Let  $\Psi=\Phi^{-1}$ and $y_0=\Phi(x_0)$ and consider the ball $B(y_0, 2R)$.
According to the inequalities in (ii) of Proposition~\ref{prop2}, we have the inclusions
$$
B(x_0, R)\subset\Psi(B(y_0, 2R))\subset B(x_0, 4R).
$$

\begin{proposition}\label{prop3}
Let $\varrho\in L^1_{loc}(\Omega)$ be a  solution to equation {\rm (\ref{eq1})}. Then the function
$$
\sigma(y)=|\det\Psi'(y)|\varrho(\Psi(y))
$$
on $B(y_0, 2R)$ satisfies the equation $\mathcal{L}^{*}\sigma=0$, where
$$
\mathcal{L}f(y)=q^{km}(y)\partial_{y_k}\partial_{y_m}f(y)+h^{k}(y)\partial_{y_k}f(y)
$$
and the coefficients have the form
$$
q^{km}(y)=a^{ij}(\Psi(y))\partial_{x_i}\Phi^k(\Psi(y))\partial_{x_j}\Phi^m(\Psi(y)), \quad
h^k(y)=\lambda u^k(\Psi(y)).
$$
\end{proposition}
\begin{proof}
The function $\varrho$ belongs to  $L^{r}(B(x_0, 4R))$ for all $1\le r<d/(d-1)$
(see the comments at the beginning of the proof of \cite[Theorem 2.1]{BSH17}).
In particular, for $p>d$ the value $p/(p-1)$ is less than $d/(d-1)$ and the inclusion $\varrho\in L^{p/(p-1)}(B(x_0, 4R))$ holds.
Therefore, in the integral equality determining the solution, in place of test functions of class
$C_0^{\infty}(\Omega)$ we can substitute test  functions of class $W^{p,2}(\Omega)$ with compact support in $B(x_0, 4R)$.
For every function $\varphi\in C_0^{\infty}(B(y_0, 2R))$ (outside the ball $B(y_0, 2R)$ we always extend $\varphi$
by zero), the function $\varphi\circ\Phi$ belongs to $W^{2, p}(B(x_0, 4R))$, has compact support in $B(x_0, 4R)$
and satisfies the equality
\begin{multline*}
\int_{B(x_0, 4R)}\Biggl(\Bigl[a^{ij}(x)\partial_{x_i}\Phi^k(x)\partial_{x_j}\Phi^m(x)\Bigr]\partial_{y_k}\partial_{y_m}\varphi(\Phi(x))
\\
+\Bigl[a^{ij}(x)\partial_{x_i}\partial_{x_j}\Phi^k(x)+b^i(x)\partial_{x_i}\Phi^k(x)\Bigr]\partial_{y_k}\varphi(\Phi(x))\Biggr)
\varrho(x)\,dx=0.
\end{multline*}
Since $b=\beta$ on $B(x_0, 4R)$ and $\Phi^k=x_k+u^k(x)$, we have
$$
a^{ij}(x)\partial_{x_i}\partial_{x_j}\Phi^k(x)+b^i(x)\partial_{x_i}\Phi^k(x)=\lambda u^k(x).
$$
Then
$$
\int_{B(x_0, 4R)}\Bigl(q^{km}(\Phi(x))\partial_{y_k}\partial_{y_m}\varphi(\Phi(x))+
\lambda u^k(x)\partial_{y_k}\varphi(\Phi(x))\Bigr)
\varrho(x)\,dx=0.
$$
Using the change of variable $y=\Phi(x)$ and taking into account that the support of $\varphi$ belongs to $B(y_0, 2R)$, we obtain
$$
\int_{B(y_0, 2R)}\Bigl(q^{km}(y)\partial_{y_k}\partial_{y_m}\varphi(y)+
\lambda u^k(\Psi(y))\partial_{y_k}\varphi(y)\Bigr)
\varrho(\Psi(y))|\det\Psi'(y)|\,dy=0.
$$
Since $\varphi$ was arbitrary, we conclude that $\sigma(y)=\varrho(\Psi(y))|\det\Psi'(y)|$ is a solution
to the equation $\mathcal{L}^{*}\sigma=0$.
\end{proof}

Observe that the vector field $h(y)=\lambda u(\Psi(y))$ is continuously differentiable on the ball
$B(y_0, 2R)$.
In addition, the derivatives of $\Phi$ also satisfy the H\"older condition.
Therefore, the function $\sigma$ on $B(y_0, 2R)$ satisfies the equation $\mathcal{L}^{*}\sigma=0$,
in which the coefficients $q^{mk}$ of the second order terms form a nondegenerate matrix and belong to the class $VMO$ and the
coefficients $h^k$ are continuous on $B(y_0, 2R)$. This enables us to apply the results from the papers \cite{BSH17}, \cite{DongEscKim},
and \cite{Sjogan73}, \cite{Sjogan75} to the function $\sigma$ and then to transfer them to $\varrho$. Let us give an
example demonstrating a simple derivation of the known result of \cite[Theorem 3.1]{BSH17}) from the case of a nice drift.

We recall  that a mapping satisfies Dini's condition if for its modulus of continuity $\omega$ we have
$$
\int_0^1\frac{\omega(t)}{t}\,dt<\infty.
$$

\begin{example}\label{prim1}
\rm
If  conditions $\rm\bf H_a$ and $\rm\bf H_b$ are fulfilled and the matrix $A$ satisfies Dini's condition, then every solution
 $\varrho\in L^1_{loc}(\Omega)$ to equation {\rm (\ref{eq1})} has a continuous version.
\end{example}
\begin{proof}
Let us verify the existence of a continuous version of $\varrho$ on the ball $B(x_0, R/2)\subset B(x_0, 4R)\subset\Omega$.
Let $\Phi$ be the diffeomorphism constructed above.
By Proposition \ref{prop3} the function
$
\sigma(y)=\varrho(\Psi(y))|\det\Psi'(y)|
$
 satisfies on $B(y_0, 2R)$ an equation with some coefficients for which  the hypotheses of \cite[Theorem 1]{Sjogan75}
 are fulfilled, that is, the matrix $(q^{mk})$ is nondegenerate  and the functions $q^{mk}$, $h^k$ satisfy  Dini's condition.
Hence $\sigma$  has a continuous version on $B(y_0, R)$. Since $\Phi$ is a diffeomorphism of class $C^1$,
the mappings $\Phi$ and $\Psi$ take sets of measure zero to sets of measure zero and a modification of the function $\sigma$ on a  set
of measure zero yields a change of $\varrho$ on a set of measure zero. Therefore, the function $\varrho$
has a continuous version on $B(x_0, R/2)$.
\end{proof}

Note that on the ball $B(x_0, R/2)\subset B(x_0, 4R)\subset\Omega$ the modulus of
continuity of the solution $\varrho$ depends only on the quantities $d$, $p$, $R$, $\omega$, $\nu$, and $\|b\|_{L^p(B(x_0, 4R))}$.

In a similar way, by using \cite[Theorem 2]{Sjogan73} one can derive the H\"older continuity of the solution,
provided that the functions $a^{ij}$ are H\"older continuous. However, unlike
\cite[Theorem 3.1]{BSH17}, this method does not ensure the H\"older order of the
solution to be equal to the H\"older order of the matrix $A$, because the expression for the coefficients $q^{mk}$
involves the derivatives of  mapping $\Phi$, but their H\"older order depends on $d$ and $p$.

\section{Regularity of solutions}

In this section we apply Zvonkin's transform for establishing the regularity of solutions.
We first discuss the  case where the matrix $A$ satisfies the classical  Dini condition, then
consider the Dini mean oscillation condition, and finally study the integrability of solutions  without the assumption about Dini's  condition.

The next  assertion generalizes the Harnack inequality to the case where the  diffusion matrix satisfies Dini's condition
and the  drift coefficient is locally unbounded (and is merely integrable to some power larger than the dimension).
In the known results, the drift  coefficient is either zero or locally bounded, which has been substantially used in the proofs.

\begin{theorem}\label{th2}
If $a^{ij}$, $b^i$ satisfy conditions $\rm\bf H_a$ and $\rm\bf H_b$ and the matrix $A$ satisfies Dini's  condition,
then the continuous version of every nonnegative solution $\varrho\in L^1_{loc}(\Omega)$ to equation {\rm (\ref{eq1})}
satisfies the Harnack inequality, that is, for every ball $B(x_0, R/2)\subset B(x_0, 4R)\subset\Omega$ there exists a number $C$ such  that
$$
\sup_{x\in B(x_0, R/2)}\varrho(x)\le C\inf_{x\in B(x_0, R/2)}\varrho(x),
$$
where $C$ depends on $R$, $\omega$, $d$, $\nu$, $p$, and $\|b\|_{L^p(B(x_0, 4R))}$, but does not depend on the solution $\varrho$.
\end{theorem}
\begin{proof}
Let
$B(x_0, R/2)\subset B(x_0, 4R)\subset \Omega$
and let $\Phi$ be the diffeomorphism constructed above.
As above, according to Proposition \ref{prop3} the function
$$
\sigma(y)=\varrho(\Psi(y))|\det\Psi'(y)|
$$
on $B(y_0, 2R)$ satisfies the equation with coefficients for which  the hypotheses of \cite[Corollary 3.6]{BSH17} are fulfilled,
i.e., the matrix $(q^{mk})$ is nondegenerate, the functions $q^{mk}$ satisfy Dini's condition   and the
functions $h^k$ are bounded. Therefore, there exists a number $C$ depending on the objects listed above such that
$$
\sup_{y\in B(y_0, R)}\sigma(y)\le C\inf_{y\in B(y_0, R)}\sigma(y).
$$
Since $2^{-1}\le |\det\Psi'(y)|\le 2$, we have
$$
\sup_{y\in B(y_0, R)}\varrho(\Psi(y))\le 4C\inf_{y\in B(y_0, R)}\varrho(\Psi(y)).
$$
By the inclusion  $B(x_0, R/2)\subset \Psi(B(y_0, R))$ we have
$$
\sup_{x\in B(x_0, R/2)}\varrho(x)\le \sup_{y\in B(y_0, R)}\varrho(\Psi(y)), \quad
\inf_{y\in B(y_0, R)}\varrho(\Psi(y))\le \inf_{x\in B(x_0, R/2)}\varrho(x).
$$
Therefore,
$
\sup_{x\in B(x_0, R/2)}\varrho(x)\le 4C\inf_{x\in B(x_0, R/2)}\varrho(x).
$
\end{proof}

\begin{remark}\label{remarkkruj}\rm
Using the method suggested in \cite{Krujkov} and \cite{OlKruj}, increasing the dimension,
it is possible to add a potential term to the  drift coefficient. Let $\varrho$ be a solution in $\mathbb{R}^d$
to the equation
\begin{equation}\label{eq3}
\partial_{x_i}\partial_{x_j}\bigl(a^{ij}\varrho\bigr)-\partial_{x_i}\bigl(b^i\varrho\bigr)+c\varrho=0
\end{equation}
Then on $\mathbb{R}^{d+1}=\mathbb{R}^d_x\times\mathbb{R}_y^1$ the function $\varrho$ satisfies the equation
$$
\partial_{x_i}\partial_{x_j}\bigl(a^{ij}\varrho\bigr)+\partial^2_y\varrho-\partial_{x_i}\bigl(b^i\varrho\bigr)
-\partial_y(-cy\varrho)=0.
$$
Using this approach, one can apply the result obtained above  to the equation with the potential term $c\varrho$,
but in this case it is necessary to assume a higher integrability of the coefficients: $b^i, c\in L^{p}_{loc}(\Omega)$ with $p>d+1$.

Note also that the equation with a nonzero drift coefficient  $b$ can be transformed in a similar way into an equation without the drift.
Let $\varrho$ be a solution to equation~(\ref{eq1}) on $\mathbb{R}^d$. Then on $\mathbb{R}^{d+1}=\mathbb{R}^d_x\times(0, 1)$
the function $\varrho$ satisfies the  equation
$$
\partial_{x_i}\partial_{x_j}\bigl(a^{ij}\varrho\bigr)+M\partial^2_y\varrho
-\partial_y\partial_{x_i}\bigl((2+\tfrac{y}{2})b^i\varrho\bigr)
-\partial_{x_i}\partial_y\bigl((2+\tfrac{y}{2})b^i\varrho\bigr)=0.
$$
The new matrix $\widetilde{A}$ has the form
$$
\left(
  \begin{array}{cccc}
    a^{11} & a^{12} & \ldots & (2+\tfrac{y}{2})b^1 \\
    a^{21} & a^{22} & \ldots & (2+\tfrac{y}{2})b^2 \\
  \cdot & \cdot & \ldots & \cdot \\
  (2+\tfrac{y}{2})b^1 & (2+\tfrac{y}{2})b^2 & \ldots & M \\
  \end{array}
\right).
$$
For every $\xi\in\mathbb{R}^{d+1}$ we have the equality
$$
\langle\widetilde{A}\xi, \xi\rangle=\sum_{i, j=1}^da^{ij}\xi_i\xi_j+\sum_{j=1}^d(4+y)b^j\xi_j\xi_{d+1}+M\xi_d^2.
$$
If the functions $b^i$ are bounded and $A\ge m\cdot I$, then for $M$ sufficiently large the matrix $\widetilde{A}$
is positive definite. A certain drawback of such transformations is the necessity to impose on $b$ the same restrictions as on $a^{ij}$,
for example, to require the continuity and Dini's condition. In the next remark we shall show how one can accomplish smoothing of coefficients
$b$ and $c$ with the aid of renormalization of solutions and Zvonkin's transform.
\end{remark}

\begin{remark}\label{remarkpot}\rm
Let $\varrho\in L^1_{loc}(\Omega)$ be a  solution to equation (\ref{eq3}), where $c\in L^p_{loc}(\Omega)$ for some $p>d$,
the coefficients $a^{ij}$, $b^i$ satisfy conditions $\rm\bf H_a$ and $\rm\bf H_b$ and the matrix $A$ is continuous.
Let $B(x_0, 4R)\subset\Omega$ and let $\Phi$ be the diffeomorphism constructed before Proposition~\ref{prop2}.
Set $y_0=\Phi(x_0)$ and $\Psi=\Phi^{-1}$.
Similarly to Proposition \ref{prop3}, one verifies that the function
$$
\sigma(y)=|{\rm det}\Psi'(y)|\varrho(\Psi(y))|
$$
 satisfies the equation
$\mathcal{L}^{*}\sigma+g\sigma=0$ on $B(y_0, 2R)$, where
$$
\mathcal{L}f(y)=q^{km}(y)\partial_{y_k}\partial_{y_m}f(y)+h^{k}(y)\partial_{y_k}f(y)
$$
and the coefficients have the form
$$
q^{km}(y)=a^{ij}(\Psi(y))\partial_{x_i}\Phi^k(\Psi(y))\partial_{x_j}\Phi^m(\Psi(y)), \quad
h^k(y)=\lambda u^k(\Psi(y)), \quad g(y)=c(\Psi(y)).
$$
Let us observe that $h^k$ is a continuously differentiable function on $B(y_0, 2R)$, $g\in L^p(B(y_0, 2R))$,
$Q=(q^{km})$ satisfies condition $\rm\bf H_a$, and the function $q^{km}$ is continuous. Let $\gamma>0$.

Let us consider the Dirichlet problem
$$
\mathcal{L}u+(g-\gamma)u=0 \quad \hbox{\rm on} \quad B(y_0, R), \quad u=1 \quad \hbox{\rm on} \quad \partial B(y_0, R).
$$
Since we do not assume that the coefficient $g$ is bounded from above, for completeness we give a short justification
of the existence of a positive  solution under the condition
that the number $\gamma$ is sufficiently large. It is clear that it suffices to consider the Cauchy problem with zero boundary condition and some
right-hand side. Set $B=B(y_0, R)$. By \cite[Theorem 9.14 and Theorem 9.15]{GT} (see also \cite[Chapter 8]{Krbook})
there exists $\gamma_0>0$ such  that for every $\gamma>\gamma_0$ the Dirichlet problem $\mathcal{L}v-\gamma v=f$ on $B$ with $v=0$ on $\partial B$
has a solution $v\in W^{p, 2}(B)\cap W^{p, 1}_0(B)$ for every function $f\in L^p(B)$.
In addition, for all $\gamma>\gamma_0$ and $v\in W^{p, 2}(B)\cap W^{p, 1}_0(B)$ we have the estimate
$$
\gamma\|u\|_{L^p(B)}+\sqrt{\gamma}\|\nabla u\|_{L^p(B)}+\|D^2u\|_{L^p(B)}\le N_1\|\mathcal{L}v-\gamma v\|_{L^p(B)}.
$$
By the   embedding theorem $\|v\|_{L^{\infty}(B)}\le N_2\|Dv\|_{L^p(B)}$ and
$$
\|gv\|_{L^p(B)}\le N_2\|g\|_{L^p(B)}\|Dv\|_{L^p(B)}.
$$
Therefore, for $\gamma$  sufficiently large and every $v\in W^{p, 2}(B)\cap W^{p, 1}_0(B)$ we have the estimate
$$
\|v\|_{W^{p, 2}(B)}\le N_3\|\mathcal{L}v+(g-\gamma)v\|_{L^p(B)}.
$$
The method of continuation with respect to a parameter gives a  solution $v\in W^{p, 2}(B)\cap W^{p, 1}_0(B)$
to the equation $\mathcal{L}v+(g-\gamma)v=f$ for every $f\in L^p(B)$, hence a solution $u$ to the considered Dirichlet problem.
We show that it is positive. For a function $\eta$ let $\eta^{-}$ and $\eta^{+}$ be the negative and
positive parts  of $\eta$. Let $w=-u$. If $w\le 0$, then everything is proved. If $w$ is positive somewhere, then
$\sup_Bw=\sup_Bw^{+}$. Since
$$
\mathcal{L}w+(g-\gamma)^{-}w\ge -(g-\gamma)^{+}w^{+},
$$
by the   maximum principle (see \cite[Theorem 9.1]{GT}) we have
$$
\sup_B w^{+}\le N\|(g-\gamma)^{+}w^{+}\|_{L^d(B)}\le N\|(g-\gamma)^{+}\|_{L^d(B)}\sup_B w^{+}.
$$
Take  $\gamma$ so large that $N\|(g-\gamma)^{+}\|_{L^d(B)}<1$. Then $\sup_B w^{+}\le 0$. Therefore, $u\ge 0$.
The strict positivity follows from the Harnack inequality  (see, e.g., \cite{Trud80}).
Finally, we observe that by the Sobolev  embedding theorem the function $u$ has a continuously differentiable version,
moreover, $u_{x_i}$ belongs to some H\"older space.
We shall work with this version. Substituting into the integral equality
$$
\int\bigl[\mathcal{L}\varphi+g\varphi\bigr]\sigma\,dy=0
$$
in place of the function $\varphi$ the function $u\psi$, where $\psi\in C_0^{\infty}(B(y_0, R))$, we arrive at the equality
$$
\int\bigl[\mathcal{L}\psi+2\langle u^{-1}Q\nabla u, \nabla \psi\rangle+\gamma\psi\bigr]u\sigma\,dy=0.
$$
Therefore, the function $u\sigma$ is a solution to the equation $\widetilde{\mathcal{L}}^{*}(u\sigma)+\gamma(u\sigma)=0$
with the operator
$$
\widetilde{\mathcal{L}}f=\mathcal{L}f(y)+2\langle u^{-1}(y)Q(y)\nabla u(y), \nabla f(y)\rangle.
$$
Thus, after these transformations we obtain the equation in which the coefficient $c$ is constant and the coefficient $b$ is continuous
and even belongs to some H\"older class. Next, we apply the transform from Remark \ref{remarkkruj} and arrive at the equation with zero
coefficients $b$ and $c$. Moreover, if the original matrix $A$ satisfies Dini's condition, then after these transformations
it also satisfy this condition, in particular, Theorem \ref{th2} extends to equations with the zero order term $c\varrho$,
provided that $c\in L^p_{loc}(\Omega)$.
\end{remark}

As already noted above,  in \cite[Lemma 4.2]{DongEscKim} the Harnack inequality was established under the assumption
that $b=0$ and $A$ satisfies the Dini mean oscillation condition. Using   Zvonkin's transform  and the methods of killing
first and zero order terms explained above, we can generalize this assertion and obtain
an analog of Theorem \ref{th2} for the  matrix $A$ that satisfies the Dini mean oscillation condition.

Following \cite{DongKim} and \cite{DongEscKim}, we shall say that a measurable function $f$ on $\mathbb{R}^d$ satisfies
the Dini mean oscillation condition if for some $t_0>0$
$$
\int_0^{t_0}\frac{w(r)}{r}\,dr<\infty,
$$
where
$$
w(r)=\sup_{x\in \Omega}\frac{1}{|\Omega(x, r)|}\int_{\Omega(x, r)}\Bigl|f(y)-f_{\Omega}(x, r)\Bigr|\,dy,
$$
$$
f_{\Omega}(x, r)=\frac{1}{|\Omega(x, r)|}\int_{\Omega(x, r)}f(y)\,dy, \quad \Omega(x, r)=\Omega\cap B(x, r).
$$

Clearly, the classical Dini condition implies the Dini mean oscillation condition.
If there exists a number $N$ such that $|\Omega(x, r)|\ge Nr^d$
for all $x\in\Omega$ and $0<r<{\rm diam}\,\Omega$,
then, according to \cite[Lemma A1]{HwangKim}, any measurable function $f$ satisfying the Dini mean oscillation condition,
has a version uniformly continuous on $\Omega$, moreover, its modulus of continuity is estimated by
$\displaystyle\int_0^{|x-y|}\frac{w(r)}{r}\,dr$. We shall work with this continuous version.
In addition, according to \cite[Lemma 2.1]{DongEscKim}, the product $fg$ satisfies the Dini mean oscillation condition
if $g$ satisfies the classical Dini condition   and $f$ satisfies the Dini mean oscillation condition.

Suppose that a function $f$ is defined on $\mathbb{R}^d$ and
$|\Omega(x, r)|\ge Nr^d$ for all points $x\in\Omega$ and $r\in (0,{\rm diam}\,\Omega)$.
Then for all $x\in\Omega$ and some constant $C(d, N)>0$ the estimate
$$
\frac{1}{|\Omega(x, r)|}\int_{\Omega(x, r)} |f(y)-f_{\Omega}(x, r)|\,dy\le
\frac{C(d, N)}{|B(x, r)|}\int_{B(x, r)} |f(y)-f_{B}(x, r)|\,dy
$$
holds, where
$$
f_{B}(x, r)=\frac{1}{|B(x, r)|}\int_{B(x, r)}f(y)\,dy.
$$
Therefore, in order to verify the Dini mean oscillation condition it suffices to show  that
$$
\int_0^{t_0}\frac{\widetilde{w}(r)}{r}\,dr<\infty, \quad \hbox{where} \quad
\widetilde{w}(r)=\sup_{x\in\Omega}\frac{1}{|B(x, r)|}\int_{B(x, r)} |f(y)-f_{B}(x, r)|\,dy.
$$
We need the following observations.

\begin{lemma}\label{Dinisave}
{\rm (i)} Suppose that a function $f$ on $B(z, 4R)\subset\mathbb{R}^d$ satisfies  the Dini mean oscillation condition,
$\Phi\colon \mathbb{R}^d\to\mathbb{R}^d$ is a diffeomorphism of class $C^1$, the functions
$\partial_{x_i}\Phi^j$ are H\"older continuous and $2^{-1}\|x-y\|\le\|\Phi(x)-\Phi(y)\|\le 2\|x-y\|$. Then
the function $f\circ\Phi$ satisfies the Dini mean oscillation condition on $B(z', R)$, where $z'=\Phi^{-1}(z)$.

{\rm (ii)} Suppose that a function $f$ on $B(z, 4R)\subset\mathbb{R}^d$ satisfies the Dini mean oscillation condition. Then the function
$$F(x_1, \ldots, x_d, x_{d+1})=f(x_1, x_2, \ldots, x_d)$$
satisfies the Dini mean oscillation condition on $B((z, z_{d+1}), R)$ for every $z_{d+1}$.
\end{lemma}
\begin{proof}
Let $0<r<2R$ and
$$
w(r)=\sup_{x\in B(z, 2R)}\frac{1}{|B(x, r)|}\int_{B(x, r)} |f(y)-f_{B}(x, r)|\,dy.
$$
Let us prove (i). By assumption the functions $\partial_{x_i}\Phi^j$ are the H\"older continuous of some order $\gamma\in (0, 1)$.
It suffices to verify that the Dini mean oscillation condition is fulfilled for the
function $g(x)=f(\Phi(x))|{\rm det}\,\Phi'(x)|$. Let $x\in B(z', 2R)$ and $0<r<R$. We have
$$
\frac{1}{|B(x, r)|}\int_{B(x, r)}\Bigl|g(v)-g_{B}(x, r)\Bigr|\,dv\le
\frac{1}{|B(x, r)|}\int_{B(\Phi(x), 2r)}\Bigl|f(y)-|{\rm det}\,\Phi'(\Phi^{-1}(y))|^{-1}g_{B}(x, r)\Bigr|\,dy.
$$
Let us estimate the difference
$$
|{\rm det}\,\Phi'(\Phi^{-1}(y))|^{-1}g_{B}(x, r)-f_{B}(\Phi(x), 2r).
$$
Let replace $|{\rm det}\,\Phi'(\Phi^{-1}(y))|^{-1}$ by
$|{\rm det}\,\Phi'(x)|^{-1}$. The difference between these two numbers is estimated by $N_1r^{\gamma}$ with some constant $N_1>0$.
Observe that
\begin{multline*}
g_{B}(x, r)=\frac{1}{|B(x, r)|}\int_{\Phi(B(x, r))}f(y)\,dy
\\
=\frac{1}{|B(x, r)|}\int_{\Phi(B(x, r))}\bigl(f(y)-f_{B}(\Phi(x), 2r)\bigr)\,dy+f_{B}(\Phi(x), 2r)\frac{|\Phi(B(x, r))|}{|B(x, r)|},
\end{multline*}
where
\begin{multline*}
\frac{1}{|B(x, r)|}\biggl|\int_{\Phi(B(x, r))}\bigl(f(y)-f_{B}(\Phi(x), 2r)\bigr)\,dy\biggr|
\\
\le\frac{2^d}{|B(\Phi(x), 2r)|}\int_{B(\Phi(x), 2r)}\bigl|f(y)-f_{B}(\Phi(x), 2r)\bigr|\,dy\le 2^d w(2r).
\end{multline*}
In addition, for some constant $N_2>0$ we have the estimate
$$
\Bigl|1-\frac{|\Phi(B(x, r))|}{|{\rm det}\,\Phi'(x)|\,|B(x, r)|}\Bigr|\le N_2r^{\gamma}.
$$
Thus,
$$
\bigl||{\rm det}\,\Phi'(\Phi^{-1}(y))|^{-1}g_{B}(x, r)-f_{B}(\Phi(x), 2r)\bigr|
\le (N_1+N_2)r^{\gamma}+2^dw(2r).
$$
Therefore, we arrive at the inequality
$$
\frac{1}{|B(x, r)|}\int_{B(x, r)} |g(v)-g_{B}(x, r)|\,dv\le N_3(w(2r)+r^{\gamma}),
$$
which shows that $g$ satisfies the Dini mean oscillation condition.

Let us prove assertion (ii). Let $c_d$ denote the volume of the unit ball in $\mathbb{R}^d$.

Let $(x, x_{d+1})\in B((z, z_{d+1}), R)$ and $0<r<R$.
We observe that by Fubini's theorem
\begin{multline*}
F_{B}((x, x_{d+1}), r)=\frac{2}{c_{d+1}}\int_{B(0, 1)}f(x_1+ry_1,\ldots, x_d+ry_d)\sqrt{1-y_1^2-\ldots-y_d^2}\,dy_1\ldots dy_{d}
\\
=\frac{2}{c_{d+1}}\int_{B(0, 1)}\Bigl(f(x_1+ry_1,\ldots, x_d+ry_d)-f_B(x, r)\Bigr)\sqrt{1-y_1^2-\ldots-y_d^2}\,dy_1\ldots dy_{d}
+f_B(x, r),
\end{multline*}
where the first term is estimated in absolute value by $Nw(r)$ with some constant $N>0$.
We have
\begin{multline*}
\frac{1}{|B((x, x_{d+1}), r)|}\int_{B((x, x_{d+1}), r)} |F(v)-F_{B}((x, x_{d+1}), r)|\,dv
\\
\le
Nw(r)+\frac{1}{|B((x, x_{d+1}), r)|}\int_{B((x, x_{d+1}), r)} |F(v)-f_{B}(x, r)|\,dv.
\end{multline*}
Applying Fubini's theorem to the second term, we estimate it by
$$
\frac{2r}{|B((x, x_{d+1}), r)|}\int_{B(x, r)} |f(y)-f_{B}(x, r) |\,dy.
$$
Thus, for some constant $C(d, N)>0$ we obtain
$$
\frac{1}{|B((x, x_{d+1}), r)|}\int_{B((x, x_{d+1}), r)} |F(v)-F_{B}((x, x_{d+1}), r)|\,dv\le
C(d, N)w(r),
$$
which implies our claim.
\end{proof}

Our next result generalizes Example \ref{prim1} and Theorem \ref{th2} to the case where the  diffusion matrix
satisfies the Dini mean oscillation condition.

\begin{theorem}\label{th12}
Suppose that condition $\rm\bf H_a$ is fulfilled, on every ball
the matrix $A$  satisfies the Dini mean oscillation condition with some function $\omega$,
and $b^i, c\in L^{d+}_{loc}(\Omega)$. Suppose also that
$\varrho\in L^1_{loc}(\Omega)$ is a solution to the equation
$$
\partial_{x_i}\partial_{x_j}\bigl(a^{ij}\varrho\bigr)-\partial_{x_i}\bigl(b^i\varrho\bigr)+c\varrho=0.
$$
Then the function $\varrho$ has a continuous version. Moreover, if $\varrho\ge 0$,
then the continuous version of $\varrho$ satisfies the Harnack inequality, i.e., for every ball
$B(x_0, R/2)\subset B(x_0, 4R)\subset\Omega$
there exists a  number $C$ such that
$$
\sup_{x\in B(x_0, R/2)}\varrho(x)\le C\inf_{x\in B(x_0, R/2)}\varrho(x),
$$
where $C$ depends on $R$, $w$, $d$, $\nu$, $p$, $\|c\|_{L^p(B(x_0, 4R))}$,
and $\|b\|_{L^p(B(x_0, 4R))}$, and does not depend on~$\varrho$.  The modulus of continuity
of $\varrho$ on $B(x_0, R/2)$  depends on the same objects.
\end{theorem}
\begin{proof}
Justification is the same as in Example \ref{prim1} and Theorem \ref{th2}, but in place of results
from \cite{Sjogan75} and \cite{BSH17} we apply \cite[Theorem 1.10]{DongKim} and \cite[Lemma 4.2]{DongEscKim},
because Zvonkin's transform combined with Remark \ref{remarkpot}
and the transformation from Remark \ref{remarkkruj} enable us to reduce the proof to the
case of $b=0$ and $c=0$, moreover, the elements of the new matrix $A$ satisfy the Dini mean oscillation condition by Lemma \ref{Dinisave}.
\end{proof}

We now discuss the integrability of solutions without the assumption about   Dini's condition.
Let us recall that, according to \cite[Theorem 2.1]{BSH17}, if $a^{ij}$ and $b^i$
satisfy conditions $\rm\bf H_a$ and $\rm\bf H_b$ and $a^{ij}\in VMO$, then
every solution $\varrho\in L^1_{loc}(\Omega)$ is locally integrable to every power $p\ge 1$.
If the functions $a^{ij}$ satisfy the Dini mean oscillation condition, then the solution is locally bounded and even continuous.
It is of interest to study integrability when $A$ is slightly better than $VMO$, but does not satisfy  Dini's condition.

Suppose that the coefficients $a^{ij}$ and $b^i$ satisfy conditions $\rm\bf H_a$ and $\rm\bf H_b$ and that
$$
\|A(x)-A(y)\|\le\omega(\|x-y\|),
$$
where $\omega$ is an increasing continuous function on $[0, +\infty)$ and $\omega(0)=0$.
We also assume that for some $C_{\omega}>0$ and all $t\ge 0$ the inequality
$$\omega(t)\ge C_{\omega}t^{1-d/p}$$
holds. Set
$$
\Lambda(t)=-\int_{\tfrac{1}{t}}^1\ln(\omega^{-1}(s))\,ds.
$$
Integrating by parts, it is readily verified that
$$
\Lambda(\tfrac{1}{t})=\int_{\omega^{-1}(t)}^{\omega^{-1}(1)}\frac{\omega(s)}{s}\,ds-\ln\omega^{-1}(1)+t\ln\omega^{-1}(t).
$$
We observe that if   Dini's condition is fulfilled, then the function $\Lambda(t)$ is bounded from above.

\begin{theorem}\label{th4}
Let $\varrho\in L^1_{loc}(\Omega)$ satisfy equation {\rm (\ref{eq1})}.
Then, for every closed ball $B\subset\Omega$ and every $\delta>0$, there
exists a constant $c=c(B, \delta)>0$ such that for all $q\ge 1$ one has
$$
\ln\|\varrho\|_{L^p(B)}\le c+c\Lambda(cq)+\delta\ln q.
$$
\end{theorem}

We first consider the case of a bounded drift coefficient  $b$ and by  Zvonkin's transform extend the
result to the case of integrable $b$.

The next auxiliary assertion is standard, but in order to precise depends of constants on parameters
we include a justification.

\begin{lemma}\label{lemint}
Let $\alpha=(\alpha^{ij})$ be a constant  symmetric positive definite matrix such that
$$\gamma I\le \alpha\le \gamma^{-1}I.$$
Let $f\in C^{\infty}_0(B(0, 1))$ and
$$1<r<d/(d-1), \quad  1<t<\tau<d/(d-1).$$
Then there exists a function $u\in C^{\infty}(B(0, 1))$  such that
${\rm tr}(\alpha D^2u)=f$ and
\begin{equation}\label{est1}
\|u\|_{L^r(B(0, 1))}+\|\nabla u\|_{L^r(B(0, 1))}\le C(r, d, \gamma)\|f\|_{L^t(B(0, 1))},
\end{equation}
\begin{equation}\label{est2}
\|D^2u\|_{L^t(B(0, 1))}\le C(d, \tau, \gamma)t'\|f\|_{L^t(B(0, 1))},
\end{equation}
where the constants $C(d, \tau, \gamma)$ and $C(r, d, \gamma)$ do not depend on $t$.
\end{lemma}
\begin{proof}
Changing coordinates we reduce the problem to the case of the unit matrix $\alpha$.
Let $u$ be the  solution to the Dirichlet problem $\Delta u=f$ on $B(0, 2)$, $u=0$ on $\partial B(0, 2)$.
We estimate $\nabla u$. Let $g^j\in C_0^{\infty}(B(0, 2))$
and let $v$ be the solution to the Dirichlet problem $\Delta v={\rm div}\, g$ on $B(0, 2)$, $v=0$ on $\partial B(0, 2)$.
By Theorem \cite{book} and the  embedding theorem one has
$$
\sup_{B(0, 2)}|v(x)|\le C_1(r, d)\|g\|_{L^{r'}(B(0, 2))}.
$$
Then
$$
\int\langle \nabla u, g\rangle\,dx=-\int v\Delta u\,dx\le C_1(r, d)\|g\|_{L^{r'}(B(0, 2))}\|f\|_{L^1(B(0, 2))}.
$$
Since $g$ was arbitrary, we obtain the desired estimate of the norm $\|\nabla u\|_{L^r(B(0, 2))}$.
Similarly an estimate of the norm $\|u\|_{L^r(B(0, 2))}$ is obtained.

Let now $\zeta\in C_0^{\infty}(B(0, 2))$ and $\zeta=1$ on $B(0, 1)$.
Applying to  $\zeta u$ Theorems 9.8 and 9.9 from \cite{GT}, we obtain the inequality
$$
\|D^2(\zeta u)\|_{L^t(B(0, 2))}\le C_1(d, \tau)t'\|\Delta(\zeta u)\|_{L^t(B(0, 2))},
$$
where the right-hand side is estimated by
$$
C_2(d, \tau)t'\bigl(\|f\|_{L^t(B(0, 2))}+\|\nabla u\|_{L^{\tau}(B(0, 2))}+\|u\|_{L^{\tau}(B(0, 2))}\bigr).
$$
It remains to apply the   estimate for $u$ and $\nabla u$ obtained above.
\end{proof}

\begin{lemma}\label{lemint2}
Let $\varrho$ be a solution to equation {\rm (\ref{eq1})} with a bounded  drift coefficient.
Then, for every ball $B$ and every $\delta>0$, there exists  $c=c(B, \delta)>0$ such
that for all $q\ge 1$ we have
$$
\ln\|\varrho\|_{L^q(B)}\le c+c\Lambda(cq)+\delta\ln q.
$$
\end{lemma}
\begin{proof}
It suffices to obtain our estimate for large $q$.
Without loss of generality we can assume that $B=B(0, 1)$ and $B(0, 4)\subset\Omega$.
Let $x_0\in B(0, 1)$ and $0<\lambda<1$ be fixed. The function $\sigma(y)=\varrho(x_0+\lambda y)$
on the ball $B(0, 2)$ satisfies the equation
$$
\partial_{y_iy_j}\bigl(q^{ij}\sigma\bigr)-\partial_{y_i}\bigl(h^i\sigma\bigr)=0,
$$
where $q^{ij}(y)=a^{ij}(x_0+\lambda y)$ and $h^i(y)=\lambda b^i(x_0+\lambda y)$.
Let
$$
Q_0=(q_0^{ij}), \ q_0^{ij}=a^{ij}(x_0),
\ \zeta\in C_0^{\infty}(B(0, 1)), \ 0\le \zeta\le 1, \  \zeta(y)=1
\ \hbox{ if \ } y\in B(0, 1/2).
$$
By Lemma \ref{lemint}, for every function $f\in C_0^{\infty}(B(0, 1))$,
there exists a smooth solution $u$ to the equation ${\rm tr}(Q_0D^2u)=f$
satisfying estimates (\ref{est1}) and (\ref{est2}). Then
$$
\int f\zeta\sigma\,dy=\int {\rm tr}((Q_0-Q)D^2u)\zeta\sigma\,dy
-\int \bigl[u{\rm tr}(QD^2\zeta)+2\langle Q\nabla u, \nabla\zeta\rangle+\langle h,\nabla(\zeta u)\rangle\bigr]\sigma\,dy.
$$
Let $1<t<\tau<d/(d-1)$ and $1<r<d/(d-1)$. Observe that
$$
\int {\rm tr}((Q_0-Q)D^2u)\zeta\sigma\,dy\le \omega(\lambda)\|D^2u\|_{L^{t}(B(0, 1))}\|\zeta\sigma\|_{L^{t'}(B(0, 1))}.
$$
By Lemma \ref{lemint} we have $\|D^2u\|_{L^{t}(B(0, 1))}\le C_1t'\|f\|_{L^{t}(B(0, 1))}$, where $C_1$ does not depend on $t$.
Therefore,
$$
\int {\rm tr}((Q_0-Q)D^2u)\zeta\sigma\,dy\le C_1t'\omega(\lambda)\|f\|_{L^{t}(B(0, 1))}\|\zeta\sigma\|_{L^{t'}(B(0, 1))}.
$$
The expression
\begin{equation}\label{eqint}
-\int \bigl[u{\rm tr}(QD^2\zeta)+2\langle Q\nabla u, \nabla\zeta\rangle+\langle h,\nabla(\zeta u)\rangle\bigr]\sigma\,dy
\end{equation}
is estimated from above by
$$
C(\zeta)\bigl[\sup_y\|Q(y)\|+\sup_y|h(y)|\bigr]\int_{B(0, 1)}\bigl(|u(y)|+|\nabla u(y)|\bigr)|\sigma(y)|\,dy.
$$
By H\"older's inequality we obtain
$$
\int_{B(0, 1)}\bigl(|u(y)|+|\nabla u(y)|\bigr)|\sigma(y)|\,dy\le
\bigl(\|u\|_{L^r(B(0, 1))}+\|\nabla u\|_{L^r(B(0, 1))}\bigr)\|\sigma\|_{L^{r'}(B(0, 1))}.
$$
Applying again Lemma \ref{lemint}, we can estimate  (\ref{eqint}) by
$$
C_2\|f\|_{L^{t}(B(0, 1))}\|\sigma\|_{L^{r'}(B(0, 1))},
$$
where $C_2$ does not depend on $t$. Thus, we arrive at the estimate
$$
\int f\zeta\sigma\,dy\le C_1t'\omega(\lambda)\|f\|_{L^{t}}\|\zeta\sigma\|_{L^{t'}}+
C_2\|f\|_{L^{t}}\|\sigma\|_{L^{r'}(B(0, 1))}.
$$
Set now $t'=q$. Let $C_1q\omega(\lambda)=1/2$,
i.e., $\lambda=\omega^{-1}\bigl(\tfrac{1}{2C_1q}\bigr)$, where the number $q$ is so large that $\lambda<1$.
Recall that $\zeta=1$ on $B(0, 1/2)$. Therefore, we have
$$
\|\sigma\|_{L^q(B(0, 1/2))}\le 2C_2\|\sigma\|_{L^{r'}(B(0, 1))}.
$$
Let $r'<s<q$. Applying H\"older's inequality, we arrive at the estimate
$$
\|\sigma\|_{L^q(B(0, 1/2))}\le C_3\|\sigma\|_{L^{s}(B(0, 1))},
$$
where $C_3$ does not depend on $q$, $\lambda$ and $s$. Returning to the original coordinates, we obtain
$$
\|\varrho\|_{L^q(B(x_0, \lambda/2))}\le \lambda^{\tfrac{d}{q}-\tfrac{d}{s}}C_3\|\varrho\|_{L^{s}(B(x_0, \lambda))}.
$$
Let $1<R<2$. Covering the ball $B(0, R)$ by finitely many balls $B(x_i, \lambda/2)$, we arrive at the estimate
$$
\|\varrho\|_{L^q(B(0, R))}\le C_4\lambda^{-d/s}\|\varrho\|_{L^{s}(B(0, R+\lambda))},
$$
where $C_4$ does not depend on $s$, $R$ and $\lambda$. Let
$$
\beta>1, \quad \frac{\ln C_4}{\ln \beta}<\delta.
$$
Set $q_m=2C_1\beta^m$, $s=2C_1\beta^{m-1}$ and
$$
\lambda_m=\omega^{-1}\bigl(\tfrac{1}{2C_1\beta^{m}}\bigr), \quad R=r_m=2-\sum_{k=k_0}^{m}\lambda_k.
$$
Since $\omega^{-1}(s)\le \bigl(\tfrac{s}{C}\bigr)^{p/(p-d)}$, the series $\sum_{k}\lambda_k$ converges.
Further we assume that $k_0$ is so large  that $r_m>1$ for all $m>k_0$ and $\beta^{-k_0+1}<1$.
Thus, for all $m>k_0$ the inequality
$$
\|\varrho\|_{L^{q_m}(B(0, r_m))}\le e^{v_m}\|\varrho\|_{L^{q_{m-1}}(B(0, r_{m-1}))}, \quad
v_m=\delta\ln\beta-d(2C_1)^{-1}\beta^{-m+1}\ln\lambda_m
$$
holds. By iterations from $k$ to $k_0$ we obtain the inequality
$$
\ln\|\varrho\|_{L^{q_k}(B(0, r_k))}\le C_5+\sum_{m=k_0}^{k}v_m,
$$
where $C_5$ does not depend on $k$. Observe that
$$
\sum_{m=k_0}^{k}v_k\le \delta\ln q_k-d(2C_1)^{-1}\sum_{m=k_0}^{k}\beta^{-m+1}\ln\lambda_m.
$$
We have the estimate
$$
-\sum_{m=k_0}^{k}\beta^{-m+1}\ln\omega^{-1}\bigl(\tfrac{1}{2C_1\beta^{m}}\bigr)\le
-\frac{\beta}{\beta-1}\int_{\beta^{-k}}^{\beta^{-k_0+1}}\ln\omega^{-1}\bigl(\tfrac{t}{2C_1}\bigr)\,dt.
$$
Thus,
$$
\ln\|\varrho\|_{L^{q_k}(B(0, 1))}\le C_5+\delta\ln q_k+C_6\Lambda(q_k).
$$
If  now $\beta^{-1}q_{k}\le q\le q_{k}$ and $c$ is greater than $C_5$, $C_6$ and $\beta$, we
obtain the resulting estimate $\ln\|\varrho\|_{L^{q}(B(0, 1))}\le c+c\Lambda(cq)+\delta\ln q$.
\end{proof}

We now prove Theorem \ref{th4}.

\begin{proof}
It suffices to show that given
$$B(x_0, R/2)\subset B(x_0, 4R)\subset\Omega$$
and $\delta>0$, there exists a constant $c>0$ for which
$$
\ln\|\varrho\|_{L^q(B(x_0, R/2))}\le c+c\Lambda(cq)+\delta\ln q.
$$
Let $\Phi$ be the diffeomorphism constructed before Proposition \ref{prop2}.
According to Proposition \ref{prop3}, the function
$$
\sigma(y)=\varrho(\Psi(y))|\det\Psi'(y)|
$$
  on the ball $B(y_0, 2R)$ centered at $y_0=\Phi(x_0)$ is a solution to an
equation whose coefficients satisfy the hypotheses of Lemma \ref{lemint2}, that is, the
matrix $(q^{mk})$ satisfies condition $\rm\bf H_a$ and the  drift coefficient $h$ is bounded.
In addition,  $\Psi$ is Lipschitz, the functions $\partial_{x_i}\Phi$ are H\"older of order $1-d/p$
and $\omega(t)\ge Ct^{1-d/p}$. So  there exists a number $N>0$ such that
$$
\|Q(y)-Q(z)\|\le N\omega(N\|y-z\|).
$$
Observe that
$$
\Lambda_N(\tfrac{1}{t})=-\int_t^1\ln N^{-1}\omega^{-1}(N^{-1}s)\,ds\le
N\Lambda(\tfrac{N}{t})+\ln N.
$$
Therefore, for any closed ball $\overline{B}(y_0, R)$ in $B(y_0, 2R)$
and every $\delta>0$ there exists  $c>0$ such that
$$
\ln \|\sigma\|_{L^q(B(y_0, R))}\le c+c\ln N+cN\Lambda(cNq)+\delta\ln q.
$$
Since $\Phi$ is a $C^1$-diffeomorphism, an analogous estimate holds for the original function~$\varrho$.
\end{proof}

\begin{corollary}
Under the hypotheses of Theorem {\rm \ref{th4}}, for every closed ball $B\subset\Omega$
the following assertions hold.

{\rm (i)} If $\lim\limits_{t\to 0+}\omega(t)|\ln t|=0$, then
$$
\exp(\gamma_1|\varrho|^{\gamma_2})\in L^1(B)\quad \hbox{for all $\gamma_1, \gamma_2>0$}.
$$

{\rm (ii)} If the function $\omega(t)|\ln t|$ is bounded on $(0, 1]$, then exist  numbers  $\gamma_1, \gamma_2>0$ such that
$$
\exp(\gamma_1|\varrho|^{\gamma_2})\in L^1(B).
$$

{\rm (iii)} If for some $0<\beta<1$ the function $\omega(t)|\ln t|^{\beta}$ is bounded on $(0, 1]$,
then for some $\gamma>0$
$$
\exp(\gamma\bigl|\ln(|\varrho|+1)\bigr|^{\tfrac{1}{1-\beta}})\in L^1(B).
$$
\end{corollary}
\begin{proof}
Let us consider case (i), when $\omega(t)=o(|\ln t|^{-1})$. Let $\varepsilon>0$.
There is $s_0\in(0, 1)$ such  that for all $s\in(0, s_0)$ one has
$|\ln\omega^{-1}(s)|\le \varepsilon s^{-1}$, which
follows from the estimate $\omega(t)\le\varepsilon|\ln t|^{-1}$.
Therefore,
$$
\Lambda(cq)\le C_1(\varepsilon)+\varepsilon\ln q,
$$
so that
$$
\int_B|\varrho|^q\,dx\le C_2^qq^{\varepsilon q}.
$$
Let $q=\gamma_2k$ and $\varepsilon\gamma_2=1/2$. We obtain
$$
\int_B\frac{(|\varrho|^{\gamma_2})^k}{k!}\,dx\le \frac{C_3^kk^{k/2}}{k!}.
$$
By Stirling's formula this yields convergence of the series
$$
\int_B\sum_k\frac{(\gamma_1|\varrho|^{\gamma_2})^k}{k!}\,dx,
$$
which implies the integrability of $\exp(\gamma_1|\varrho|^{\gamma_2})$.
Assertion~(ii) is proved similarly.

Let us prove assertion (iii).
The boundedness of the function $\omega(t)|\ln t|^{\beta}$ implies the estimate
$|\ln \omega^{-1}(s)|\le M_1s^{-\tfrac{1}{\beta}}$ with some constant $M_1>0$.
By Theorem \ref{th4},  for all $q\ge 1$ and some $M_2>0$ we obtain the estimate
$$
\ln\|\varrho\|_{L^q(B)}\le M_2q^{\tfrac{1-\beta}{\beta}}.
$$
Chebyshev's inequality gives
$$
\ln\bigl|\{x\in B\colon |\varrho(x)|>t\}\bigr|\le -q\ln t+M_2q^{\tfrac{1}{\beta}}.
$$
Let $q=C(\beta, M')|\ln t|^{\beta/(1-\beta)}$, where $C(\beta, M_2)=(\beta/M_2)^{\tfrac{\beta}{1-\beta}}$.
There is a   number $t_0$ such that for all $t\ge t_0$ one has $q\ge 1$
and for some constant $M_3>0$ the inequality
$$
\ln\bigl|\{x\in B\colon |\varrho(x)|>t\}\bigr|\le -M_3|\ln t|^{1/(1-\beta)}
$$
holds. For any continuously differentiable function $f$ with $f'>0$ we have
$$
\int_Bf(|\varrho(x)|)\,dx\le \int_0^{t_0}\bigl|\{x\in B\colon f(|\varrho(x)|)>t\}\bigr|\,dt+
\int_{f(t_0)}^{+\infty}f'(s)\exp(-M_3|\ln s|^{1/(1-\beta)})\,ds.
$$
It is readily verified that for $f(t)=\exp(\gamma|\ln(t+1)|^{1/(1-\beta)})$ with $\gamma<M_3$
the corresponding integral in the right-hand side is finite, which gives the desired integrability.
\end{proof}

\section{Existence and uniqueness of probability solutions}

In this section we apply the results obtained above for constructing  positive and probability solutions to the Kolmogorov  equation.

Our next theorem   generalizes assertion (i) in Theorem 2.4.1 in \cite{book} to the case where the coefficients $a^{ij}$ satisfy
the Dini mean oscillation condition. It was assumed in the cited theorem that the functions $a^{ij}$ belong locally to the Sobolev class
$W^{p, 1}$ with $p>d$.

\begin{theorem}\label{th5}
Suppose that the coefficients $a^{ij}$ and $b^i$ are defined on all of $\mathbb{R}^d$, $b^i\in L^{d+}_{loc}$,
and for every ball $B$ we can find a   number $\nu_B>0$ and continuous nonnegative increasing function $w_B$ on $[0, 1]$ such
that $w_B(0)=0$, the integral $\displaystyle\int_0^{1}\frac{w_B(t)}{t}\,dt$ converges and
$$
\nu_B\cdot I\le A(x) \le \nu_B^{-1}\cdot I, \quad
\sup_{x\in B}\frac{1}{|B(x, r)|}\int_{B(x, r)} |a^{ij}(y)-a^{ij}_B(x, r)|\,dy\le w_B(r), \quad r\in(0, 1].
$$
Then there exists a continuous and positive solution $\varrho$ to equation  {\rm (\ref{eq1})} on $\mathbb{R}^d$.
\end{theorem}
\begin{proof}
Let $\zeta\in C_0^{\infty}(\mathbb{R}^d)$ have support in $B(0, 1)$, $\zeta\ge 0$ and
$\|\zeta\|_{L^1(\mathbb{R}^d)}=1$. For every $k\in\mathbb{N}$ let $\zeta_k$ denote the function $k^d\zeta(kx)$ and let
$$
a^{ij}_k=a^{ij}*\zeta_k, \quad b^i_k=b^i*\zeta_k.
$$
The functions $a^{ij}_k$ and $b^i_k$ are infinitely differentiable and on every ball $B_R=B(0, R)$ we have
$$
\nu_{B_{R+1}}\cdot I\le A_k\le \nu_{B_{R+1}}^{-1}\cdot I,
$$
$$
\|b_k\|_{L^{p_R}(B_{R})}\le \|b\|_{L^{p_R}(B_{R+1})}, \quad p_R=p(B_{R+1}).
$$
Moreover,
$$
\sup_{x\in B_R}\frac{1}{|B(x, r)|}\int_{B(x, r)} |a^{ij}_k(y)-{a^{ij}_k}_B(x, r)|\,dy\le w_{B_{R+1}}(r).
$$
According to \cite[Theorem 2.4.1]{book}, there exists a positive smooth solution $\varrho_k$ to the equation
$$
\partial_{x_i}\partial_{x_j}\bigl(a^{ij}_k\varrho_k\bigr)-\partial_{x_i}\bigl(b^i_k\varrho_k\bigr)=0.
$$
Multiplying by a constant, we can assume that $\varrho_k(0)=1$.
Applying the Harnack inequality from Theorem \ref{th12},  for every ball we obtain the inequality
$$
\sup_{B_R}\varrho_k(x)\le C_R\varrho_k(0)=C_R,
$$
where the constant $C_R$ does not depend on $k$. By Theorem \ref{th12} the modulus of continuity of $\varrho_k$ on every ball $B_R$
is estimated by some function $\omega_R$ that does not depend on $k$, is continuous at zero and $\omega_R(0)=0$.
Thus, on every ball the sequence $\{\varrho_k\}$ is uniformly bounded and equicontinuous,
hence, has a subsequence that converges to some nonnegative continuous function $\varrho$
uniformly on every ball. In addition, on every ball $B_R$ the sequence $\{a^{ij}_k\}$ converges uniformly
to $a^{ij}$ and the sequence $\{b^i_k\}$ converges to $b^i$ in $L^{p_R}(B_R)$. Passing to the limit in the integral identities
determining the solutions $\varrho_k$, we obtain an analogous integral equality for~$\varrho$. Therefore, the function $\varrho$
is a solution to equation (\ref{eq1}). Since $\varrho(0)=1$ and the function $\varrho$ satisfies the Harnack inequality by Theorem \ref{th12},
 it is positive.
\end{proof}

As a  corollary we obtain a generalization of the Hasminskii theorem,
which follows from Theorem \ref{th5} and \cite[Corollary 2.3.3]{book}.

\begin{corollary}
If in addition to the hypotheses of Theorem {\rm\ref{th5}} there is a function $V$ of class
$W^{d, 2}_{loc}(\mathbb{R}^d)$ along with numbers $C>0$ and $R>0$ for which
$$
\lim_{|x|\to+\infty}V(x)=+\infty, \quad LV(x)\le -C \quad \hbox{\rm if} \quad |x|>R,
$$
then there exists a continuous positive probability solution $\varrho$
to equation~{\rm (\ref{eq1})} on $\mathbb{R}^d$.
\end{corollary}

Note that assertion (ii) in \cite[Theorem 2.4.1]{book} gives a nonzero nonnegative solution
to equation (\ref{eq1}) under the assumption that the matrix $A$ is locally
positive definite and bounded and the drift  coefficient $b$ is locally bounded. If, in addition, there is
a Lyapunov function, then there exists a probability solution. In the results obtained above the condition
on $A$ is stronger and the condition on $b$ is weaker.
Let us show that the condition on $A$ can be weakened to the inclusion in the class VMO if the drift $b$ is dissipative
(which is stronger than the existence of a Lyapunov function).

\begin{theorem}\label{th51}
Suppose that the coefficients $a^{ij}$ and $b^i$ are defined on all of $\mathbb{R}^d$, $b^i\in L^{d+}_{loc}$,
and there exist numbers $\nu>0$ and $M>0$ and an increasing continuous function $\omega$ on $[0, +\infty)$ such that $\omega(0)=0$ and
$$
\nu\cdot I\le A(x)\le\nu^{-1}\cdot I, \quad
\sup_{z\in\mathbb{R}^d}r^{-2d}\int_{B(z, r)}\int_{B(z, r)}|a^{ij}(x)-a^{ij}(y)|\,dxdy\le\omega(r).
$$
Suppose also that
$$
\lim\limits_{|x|\to\infty}\langle b(x), x\rangle=-\infty.
$$
Then there exists a probability solution $\varrho$ to equation {\rm (\ref{eq1})} on $\mathbb{R}^d$.
\end{theorem}
\begin{proof}
Let $\zeta\in C_0^{\infty}(\mathbb{R}^d)$ have support in $B(0, 1)$, $\zeta\ge 0$ and
$\|\zeta\|_{L^1(\mathbb{R}^d)}=1$. For each $k\in\mathbb{N}$ let $\zeta_k(x)=k^d\zeta(kx)$ and
$a^{ij}_k=a^{ij}*\zeta_k$. The functions  $a^{ij}_k$ are infinitely differentiable and
$$
\nu\cdot I\le A_k(x)\le\nu^{-1}\cdot I, \quad
\sup_{z\in\mathbb{R}^d}r^{-2d}\int_{B(z, r)}\int_{B(z, r)}|a_k^{ij}(x)-a_k^{ij}(y)|\,dxdy\le\omega(r).
$$
In addition, for the function $V(x)=|x|^2/2$ we have
$$
{\rm tr}\bigl(A_k(x)D^2V(x)\bigr)+\langle b(x), \nabla V(x)\rangle\le \nu^{-1}d+\langle b(x), x\rangle\to-\infty.
$$
Therefore, for every $k$ there exists a probability solution $\varrho_k$ of the equation
$$
\partial_{x_i}\partial_{x_j}\bigl(a^{ij}_k\varrho_k\bigr)-\partial_{x_i}\bigl(b^i\varrho_k\bigr)=0.
$$
According to \cite[Theorema 2.1]{BSH17}, for every ball $B$ and every $s>1$ there exists a
number $C(s, B)>0$ independent of $k$ such that $\|\varrho_k\|_{L^s(B)}\le C(s, B)$.
In addition, according to \cite[Theorem 2.3.2]{book} one has
$$
\int_{\mathbb{R}^d}|\langle b(x), x\rangle|\varrho_k(x)\,dx\le
d\nu^{-1}+2\int_{|x|\le R}|\langle b(x), x\rangle|\varrho_k(x)\,dx,
$$
where $R>0$ is taken such that $\langle b(x), x\rangle<0$ whenever $|x|>R$. Observe that
$$
\int_{|x|\le R}|\langle b(x), x\rangle|\varrho_k(x)\,dx\le C(p', B_R)\|b\|_{L^p(B_R)}, \quad p'=p/(p-1), \quad B_R=B(0, R).
$$
Therefore, the sequence of measures $\varrho_k\,dx$ contains a subsequence  converging
weakly to some probability measure $\mu$. Since $\|\varrho_k\|_{L^s(B)}\le C(s, B)$, we can assume that this
subsequence converges weakly in $L^{p'}(B)$ for every ball $B$. Hence the measure $\mu$ has a density
$\varrho\in L^{p'}_{loc}(\mathbb{R}^d)$ and for every function $\varphi\in C_0^{\infty}(\mathbb{R}^d)$ we can pass to the limit
as $k\to\infty$ in the equality
$$
\int_{\mathbb{R}^d}\bigl[a^{ij}_k\partial_{x_i}\partial_{x_j}\varphi
+b^i\partial_{x_i}\varphi\bigr]\varrho_k(x)\,dx=0.
$$
The function $\varrho$ is a probability solution of equation (\ref{eq1}).
\end{proof}

Zvonkin's transform enables us to obtain the following modification of the Hasminskii theorem generalizing
\cite[Theorem 4.10(iii)]{ZZ18} to the matrix $A$ of class VMO.

\begin{corollary}\label{corzvon}
Suppose that the coefficients $a^{ij}$ and $b^i$ are defined on all of $\mathbb{R}^d$ and we can find a
number $\nu>0$ and an increasing continuous function $\omega$ on $[0, +\infty)$ such that $\omega(0)=0$ and
$$
\nu\cdot I\le A(x)\le\nu^{-1}\cdot I, \quad
\sup_{z\in\mathbb{R}^d}r^{-2d}\int_{B(z, r)}\int_{B(z, r)}|a^{ij}(x)-a^{ij}(y)|\,dxdy\le\omega(r).
$$
Let also $b^i=b_1^i+b_2^i$, where $b_2\in L^p(\mathbb{R}^d)$ with $p>d$ and
$$
\lim_{|x|\to\infty}\langle b_1(x), x\rangle\le C_1-C_2|x|^{\kappa+1}, \quad |b_1(x)|\le C_3(1+|x|^{\kappa}),
$$
for some numbers $C_1, C_2, C_3>0$ and $\kappa>0$. Then there exists a probability solution $\varrho$ to equation {\rm (\ref{eq1})} on $\mathbb{R}^d$.
\end{corollary}
\begin{proof}
Applying Proposition \ref{prop1}, for every $\delta>0$ we construct a   mapping $u=(u^1,\ldots, u^d)$ such that $u^k$ is a solution to the equation
$$
{\rm tr}(AD^2u^k)+\langle b_2, \nabla u^k\rangle-\lambda u^k=-b_2^k,
$$
$u^k\in L^{\infty}(\mathbb{R}^d)$ and $\sup_x|\nabla u^k(x)|<\delta$.
According to Proposition \ref{prop2}, for $\delta>0$ sufficiently small the mapping $\Phi(x)=x+u(x)$ is a diffeomorphism of $\mathbb{R}^d$.
Let $\Psi=\Phi^{-1}$. It is readily seen that $|\Psi(y)|\to\infty$ as $|y|\to\infty$. Let us consider the equation
\begin{equation}\label{eqq1}
\partial_{y_k}\partial_{y_m}\bigl(q^{km}\sigma\bigr)-\partial_{y_k}\bigl(h^k\sigma\bigr)=0,
\end{equation}
where
$$
q^{km}(y)=a^{ij}(\Psi(y))\partial_{x_i}\Phi^k(\Psi(y))\partial_{x_j}\Phi^m(\Psi(y)),
$$
$$
h(y)=\lambda u(\Psi(y))+b_1^i(\Psi(y))\partial_{x_i}\Phi(\Psi(y)).
$$
Since $\Phi(x)=x+u(x)$, we have
$$
\langle h(y), y\rangle(\Phi(x))=\lambda\langle u(x), x+u(x)\rangle+
\langle b_1(x), x+u(x)\rangle+
\langle b_1^i(x)\partial_{x_i}u(x), x+u(x)\rangle.
$$
Let $|u(x)|\le M$ and $|\nabla u(x)|\le \delta$. Then
$$
\langle h(y), y\rangle(\Phi(x))\le \lambda M^2+\lambda M|x|+M|b_1(x)|+\delta|b_1(x)||x|+\delta M|b_1(x)|
+\langle b_1(x), x\rangle.
$$
Using our condition on $b_1$, we can estimate the right side by
$$
C_1+\lambda M^2+\lambda M|x|+(1+\delta)MC_3(1+|x|^{\kappa})+\delta C_3(1+|x|^{\kappa+1})-C_2|x|^{\kappa+1}.
$$
Let $\delta C_3<C_2$. Then $\langle h(y), y\rangle\to -\infty$ as $|y|\to\infty$.
Therefore, the hypotheses of Theorem \ref{th51} are fulfilled and there exists a probability solution
$\sigma$ of equation (\ref{eqq1}). It is clear that $\varrho(x)=\sigma(\Phi(x))|{\rm det}\Phi'(x)|$ is a probability solution to equation (\ref{eq1}).
\end{proof}

In connection with the existence theorems proved above we discuss a probabilistic interpretation of
probability solutions of the  Kolmogorov equation. Let $a^{ij}$ and $b^i$ be Borel functions on $\mathbb{R}^d$ and let $\mu$ be a Borel probability
 measure satisfying the stationary  Kolmogorov equation
$$
\partial_{x_i}\partial_{x_j}\bigl(a^{ij}\mu\bigr)-\partial_{x_i}\bigl(b^i\mu\bigr)=0,
$$
understood in the  sense of the integral equality
$$
\int_{\mathbb{R}^d}\bigl[a^{ij}\partial_{x_i}\partial_{x_j}\varphi
+b^i\partial_{x_i}\varphi\bigr]\,d\mu=0 \quad \forall \varphi\in C_0^{\infty}(\mathbb{R}^d).
$$

\begin{theorem}\label{th6}
Suppose that $a^{ij}, b^i\in L^1_{loc}(\mu)$ and
$$
\int_{\mathbb{R}^d}\frac{\|A(x)\|+|\langle b(x), x\rangle|}{1+|x|^2}\,\mu(dx)<\infty.
$$
Then for every $T>0$ one can find a probability space $(Q, \mathcal{F}, P)$, a filtration  $\mathcal{F}_t$, and
a continuous random process $\xi_t$ and a Wiener process $w_t$ adapted to the filtration $\mathcal{F}_t$  such that
$$
d\xi_t=b(\xi_t)\,dt+\sqrt{2A(\xi_t)}dw_t
$$
on $[0, T]$ and the distribution of the random variable $\omega\to \xi_t(\omega)$ equals  $\mu$ for all $t$.
Moreover, if it is known that for every probability measure $\sigma$ on~$\mathbb{R}^d$
there exists a unique weak solution $\xi_t^{\sigma}$ of the indicated stochastic equation
on $[0, +\infty)$ with initial distribution $\sigma$ and $P_{\sigma}$ is a distribution of $\xi^{\sigma}$,
then the  measure $\mu$ is invariant for the semigroup
$$
T_tf(x)=\int_{C([0, +\infty), \mathbb{R}^d)}f(\xi_t)\, P_{\delta_x}(d\xi)
$$
on the space of bounded continuous functions, that is, for every bounded continuous function $f$ and every $t\ge 0$ we have the identity
$$
\int_{\mathbb{R}^d}T_tf(x)\mu(dx)=\int_{\mathbb{R}^d}f(x)\mu(dx).
$$
\end{theorem}
\begin{proof}
By the Ambrosio-Figalli-Trevisan superposition principle \cite{BRSH-sup}
there exists a Borel probability  measure $P_{\mu}$ on $C([0, T], \mathbb{R}^d)$
satisfying the following conditions: (i) $P_{\mu}(\{\xi\colon\, \xi_t\in Q\})=\mu(Q)$ for every Borel set $Q$
and every $t\in[0, T]$, (ii) for every function $f\in C_0^{\infty}(\mathbb{R}^d)$ the mapping
$$
(\xi, t)\to f(\xi_t)-f(\xi_0)-\int_0^t\bigl[a^{ij}(\xi_s)\partial_{x_i}\partial_{x_j}f(\xi_s)
+b^i(\xi_s)\partial_{x_i}f(\xi_s)\bigr]\,ds
$$
is a martingale with respect to $P_{\mu}$ and the filtration $\sigma(\xi_s, s\le t)$.
According to \cite[Proposition 2.1, Ch.~IV]{IkeW}, one can find a probability space $(Q, \mathcal{F}, P)$, a filtration $\mathcal{F}_t$,
and a continuous random process $\xi_t$ and a Wiener process $w_t$ adapted to $\mathcal{F}_t$ such that
$$
d\xi_t=b(\xi_t)\,dt+\sqrt{2A(\xi_t)}dw_t
$$
on $[0, T]$ and the distribution of $\xi_{t}$ coincides with $P_{\mu}$, in particular, the one-dimensional distribution
of the process $\xi_t$ does not depend on $t$ and equals $\mu$.

Assume now that the given stochastic equation has a unique weak solution and $P_{\delta_x}$
is the distribution of the solution with initial condition $\delta_x$.
We observe that due to our assumption about the uniqueness of  solutions the
mapping $\sigma\mapsto P_{\sigma}$ is Borel measurable when the  spaces of measures are equipped with their weak topologies
(or with metrics generating them). Indeed, even without any assumptions about uniqueness, the set of all pairs $(P,\sigma)$
of probability measures for which $P$ is a  measure  on the space $C([0,+\infty),\mathbb{R}^d)$ of continuous paths
and $\sigma$ is a measure on $\mathbb{R}^d$ such that $\sigma$ is the image $P$ under the mapping $x\mapsto x(0)$ and the process
$$
f(x(t))-f(x(0))-\int_0^t Lf(x(s))\, ds
$$
is a martingale with respect to the measure $P$ for all functions $f\in C_b^2(\mathbb{R}^d)$, is a Borel set in the product
$\mathcal{P}(C([0,+\infty),\mathbb{R}^d))\times \mathcal{P}(\mathbb{R}^d)$. This is seen from the fact that this set
can be described by a countable number of equalities of the form
$$
\int \psi_i\, dP=0
$$
with some countable collection of bounded Borel functions $\psi_i$ on the path space along with the equality
of the measure $\sigma$ to the image of $P$ with respect to the operator $x\mapsto x(0)$. A~countable collection $\{\psi_i\}$
arises because  for verification of the martingale property we can use a countable collection of functions~$f$,
moreover, the martingale property itself can be verified only for rational times, and the comparison of the
corresponding conditional expectations also employs a countable collection of functions.

According to \cite[Theorem 5.3, Ch.~IV]{IkeW} measures $P_{\delta_x}$ form a Markov family and $T_tf$ is a semigroup
the space of bounded continuous functions. Note that for every cylindrical set $C$ one has
$$
P_{\mu}(C)=\int_{\mathbb{R}^d}P_{\delta_x}(C)\, \mu(dx),
$$
which implies the equality
$$
\int_{C([0, +\infty), \mathbb{R}^d)}f(\xi_t)\, P_{\mu}(d\xi)=
\int_{\mathbb{R}^d}\int_{C([0, +\infty), \mathbb{R}^d)}f(\xi_t)\, P_{x}(d\xi)\, \mu(dx)
$$
for every bounded continuous function $f$. It remains to observe that the left side of the last equality
is the integral of $f$ against the measure $\mu$ and the right side is the integral of $T_tf$ against the measure $\mu$.
\end{proof}

In the general case the stationary  Kolmogorov equation  can have
several different probability solutions (see \cite[Chapter 4]{book}). Sufficient conditions for uniqueness of probability
solutions have been  obtained in \cite{BSH-DAN21} in the case where the matrix $A$ satisfies the classical  Dini  condition, but by
Theorem \ref{th12} and the results in \cite{DongEscKim} their justification extends without changes to the case
of the matrix $A$ satisfying the Dini mean oscillation condition. Thus, the following assertion is true.

\begin{theorem}\label{th7}
Let $a^{ij}$ and $b^i$ be defined on all of $\mathbb{R}^d$,
$b^i\in L^{d+}_{loc}$, and for every ball $B$ one can find
a number $\nu_B>0$ and a continuous nonnegative increasing function $w_B$ on $[0, 1]$ such
that $w_B(0)=0$, the integral $\displaystyle\int_0^{1}\frac{w_B(t)}{t}\,dt$ converges, and
$$
\nu_B\cdot I\le A(x) \le \nu_B^{-1}\cdot I, \quad
\sup_{x\in B}\frac{1}{|B(x, r)|}\int_{B(x, r)} |a^{ij}(y)-a^{ij}_B(x, r)|\,dy\le w_B(r), \quad r\in(0, 1].
$$
 Suppose that $\varrho$ is a probability solution to equation {\rm(\ref{eq1})} such that at least one of following conditions is fulfilled:

 {\rm (i)}  $(1+|x|)^{-2}a^{ij}, (1+|x|)^{-1}b^i\in L^1(\varrho\,dx)$,

 {\rm (ii)} there exists $V\in C^2(\mathbb{R}^d)$ with $\lim\limits_{|x|\to\infty}V(x)=+\infty$ and $LV\le C_1+C_2V$.

Then $\varrho$ is a unique probability solution.
\end{theorem}

Applying  Zvonkin's transform, we obtain the following sufficient  condition for uniqueness that agrees with Corollary \ref{corzvon}.

\begin{corollary}
Suppose that the coefficients $a^{ij}$ and $b^i$ are defined on all of $\mathbb{R}^d$ and we can find a
 number $\nu>0$ and an increasing continuous function $\omega$ on $[0, +\infty)$ such that
$$
\nu\cdot I\le A(x)\le\nu^{-1}\cdot I, \quad
\sup_{x\in\mathbb{R}^d}\frac{1}{|B(x, r)|}\int_{B(x, r)}|a^{ij}(y)-a^{ij}_B(x, r)|\,dy\le\omega(r),
$$
 $\omega(0)=0$ and the integral $\displaystyle\int_0^1\frac{\omega(r)}{r}\,dr$ converges.

Let also $b^i=b_1^i+b_2^i$, where $b_2\in L^p(\mathbb{R}^d)$ with $p>d$ and
$$
\lim\limits_{|x|\to\infty}\langle b_1(x), x\rangle\le C_1-C_2|x|^{\kappa+1}, \quad |b_1(x)|\le C_3(1+|x|^{\kappa}),
$$
for some $C_1, C_2, C_3>0$ and $\kappa>0$. Then there exists a unique probability
solution $\varrho$ to equation {\rm (\ref{eq1})} on $\mathbb{R}^d$.
\end{corollary}
\begin{proof}
The existence is proved in Corollary \ref{corzvon}.
After the change of coordinates as in the proof of Corollary \ref{corzvon} the coefficients of the new equation
satisfy condition (ii) in Theorem \ref{th7}, hence a probability solution is unique.
\end{proof}

We do not know whether the  continuity of the matrix $A$ is sufficient for the uniqueness of a probability solution.

This research is supported by the CRC 1283 at Bielefeld
University, the Russian Foundation for Basic Research Grant 20-01-00432,
Moscow Center of Fundamental and Applied Mathematics, and 
the Simons-IUM fellowship. S.V.~Shaposhnikov is
a winner of the contest ``Young Mathematics of Russia'', and thanks its jury and sponsors.

\end{document}